\documentclass[12pt,psfig,reqno]{amsart}
\usepackage{amssymb,amsfonts,latexsym}
\usepackage{graphics,verbatim}
\usepackage{graphicx}
\usepackage[usenames]{color} %\color{Red}
\usepackage{soul} %\textst{word}

\hoffset=-1cm \errorcontextlines=0 \numberwithin{equation}{section}

 \theoremstyle{plain}
\newtheorem{theorem}{Theorem}[section]

\newtheorem{corollary}[theorem]{\bf{Corollary}}

\date {}

\begin{document}

\title[Finite time blowup]{Finite time blowup of generalized Euler ODE in matrix geometry}

\author{Jiaojiao Li and Li Ma}

\address{Jiaojiao Li, Department of mathematics \\
Henan Normal university \\
Xinxiang, 453007 \\
China}

\email{lijiaojiao8219@163.com}

\address{Li Ma, Department of mathematics \\
Henan Normal university \\
Xinxiang, 453007 \\
China}

\email{lma@tsinghua.edu.cn}

\thanks{ The research is partially supported by the National Natural Science
Foundation of China (No.11271111)}

\begin{abstract}
In this paper, we study the finite time blowup of the generalized Euler ODE in the matrix geometry. We can extend Sullivan's result which is about the finite time blowup result of initial invertible linear operators to singular linear operators. we can give a complete answer to the question of Sullivan in the case when the initial matrix A is symmetric in the finite dimensional vector space W. Some open questions are proposed in the last part of the paper.


\end{abstract}

\maketitle

\section{main result and its proof} \label{sect1}

Motivated by the Euler ODE of incompressible frictionless fluid dynamics expressed in
terms of the vorticity variable, D.Sullivan \cite{S} has considered the evolution of vorticity described by an ODE $dX/dt=Q(X)$, where $Q:V\to V$ is a quadratic mapping from the (infinite dimensional) vector space $V$ and vorticity variable $X=X(t)\in V$. Assuming that $V=End(W)$ for some other linear space $W$, he can show that for each $Q$ outside
a proper algebraic subvariety of quadratic mappings from V to V , the ODE
$dX/dt = Q(X) $ exhibits finite time blowup for some initial condition.
In particular, he can show the finite time blowup for the system
\begin{equation}\label{euler}
dX/dt=X.X
\end{equation}
with the initial linear operator $A$ at $t=0$ being invertible. Namely, the solution is given by
$$
X(t)=-Id/(tId-Id/A),
$$
which blows up at a finite time iff the spectrum of $A$ contains a real number. Here $Id\in V=End(W)$ is the identity map.
In the last sentence of Sullivan's paper, he invites people consider the interesting question about how likely is it in the variable A for the ODE (\ref{euler}) to have finite time blowup. We point out that the invertible condition in Sullivan's result for (\ref{euler}) can be removed by defining the solution by the expression
\begin{equation}\label{sullivan}
X(t)=A/(Id-tA),
\end{equation}
where $A$ is the initial linear operator. Then the solution blows up at a finite time iff the spectrum of $A$ contains a non-zero real number.
Then, unless $W$ is an odd dimensional vector space for an open set of initial conditions
there is a solution for all time and for another open set of initial conditions there
is finite time blowup. This phenomenon is similar to the invariant sets defined by the double well potential method for the classical parabolic PDE \cite{M}.

Given a linear operator $B\in V=End(W)$. We consider the generalized Euler
ODE systems
\begin{equation}\label{li1}
dX/dt=B.X.X
\end{equation}
and
\begin{equation}\label{li2}
dX/dt=X.B.X.
\end{equation}

It is quite clear for us that the solutions to (\ref{li1}) and (\ref{li2})
can be given by the expression
\begin{equation}\label{sullivan2}
X(t)=A/(Id-tAB),
\end{equation}
provided the initial linear operator $A$ commutes with the linear operator $B$.

We notice that in the finite dimensional matrix geometry, we can give the precise blowup time for some initial conditions.

Our main result is the following.

\begin{theorem} If the initial linear operator $A$ satisfying $AB=BA$ for the ODE systems (\ref{li1}) and (\ref{li2}), then the solution with initial condition $A$ can be given by (\ref{sullivan2}) and it blows up at a finite time iff the spectrum of $A$ contains a non-zero real number. In the case when $W$ is a finite dimensional vector space with inner product $<.,.>$ and $A$ and $B$ are symmetric matrices with $AB=BA$, the blowup time of the solution is determined by the number
$$
T=(\max_{\{x; <x,x>=1\}}<ABx,x>)^{-1}
$$
and when $T\leq 0$, the solution is eternal.
\end{theorem}
\begin{proof} The first part has been proven as above. We need only to prove the last part of the result. We may only consider the case when $T>0$ and study the forward flow $X(t)$ (otherwise we consider backward flow). By Courant's minimax variational principle, we know that there is a unit vector $w\in W$ such that $ABw=T^{-1} w$.
It is clear that the solution is given by (\ref{sullivan2}) before the time $T$.

If the solution $X(t)$ exists at time $t=T$, then
$$
Aw=X(T)(Id-TAB)w=0,
$$
which is absurd with the fact $T\not=0$ (which implies that $BAw\not=0$).
This completes the proof of the theorem.
\end{proof}

In the special case when $B=Id$, we have
\begin{corollary}
When the initial matrix $A$ of the ODE (\ref{li1}) or(\ref{li2}) is symmetric in the finite dimensional vector space $W$, the solution with initial condition $A$ can be given by (\ref{sullivan2}) and it blows up at a finite time iff the spectrum of $A$ contains a non-zero real number.
\end{corollary}
Thus we can give a complete answer to the question of Sullivan in the case when the initial matrix $A$ is symmetric on the finite dimensional vector space $W$.

We remark that similar result holds true for the ODE system
$$
dX/dt=X.X.B.
$$
We leave the case when $BA\not=AB$ as an open question.

\section{Discussions and open questions} \label{sect2}

The method above may not be useful in the understanding of the ODE system
$$
dX/dt=|X|X
$$
or the system
\begin{equation}\label{jiao}
dX/dt=[B,X]X
\end{equation}
where $|X|$ is the norm of the linear operator $X$ and $[B,X]=B.X-X.B$ is the Lie bracket.
However, the method above could be useful in the study of many ODE systems of the form
\begin{equation}\label{euler2}
dX/dt=F(X),
\end{equation}
with the initial condition $A$ at $t=0$, where $F(X)$ is an analytic matrix function of the linear operator $X$. As above we may define the solution in the form
$$
X(t)=A/(Id-G(t,A))
$$
for some analytic matrix function $G(t,x)$ with $G(0,X)=0$. One can define
$$
\lambda (t,A)=\max_{\{<x,x>=1\}}<G(t,A)x,x>.
$$
and $\lambda (t,A)$ is a continuous function in $t$ with $\lambda (0,A)=0$. If there is a minimum $T>0$ such that
$$
\lambda (T,A)=1,
$$
one may show that $T$ is the finite blowup time for the ODE system (\ref{euler2}). We leave this as an open question.
It is also a very interesting question to understand the ODE system (\ref{jiao}). Another interesting topic is to know the finite time blowup or global behavior of solutions of the nonlinear "parabolic" model
$$
u_t=-\Delta u+u.u, \; \; u=u(t)\in M_n,
$$
with initial data $u(0)=u_0\in M_n$, where the Laplacian operator $\Delta$ on the matrix algebra $M_n$ is defined in \cite{L}.

\end{document}